\documentclass[12pt,a4paper]{amsart}
\usepackage{amsthm,verbatim,enumerate,amssymb}
\def\C{\mathcal C}
\def\Rz{\mathbb{R}}
\def\ff{\varphi}
\def\adj{\operatorname{adj}}
\def\cit#1#2{\cite[#1]{#2}}

\def\mo-{}
\def\sgn{\operatorname{sgn}}

\def\Nabla{D}

\def\amo-{}

\def\e{\mathbf e}

\def\ep{\varepsilon}
\def\er{\mathbb R}
\def\rn{\mathbb R^n}

\def\hauso{\mathcal H^{1}}
\def\haus{\mathcal H^{n{-}1}}
\def\hausk{\mathcal H^{k}}
\def\ve{\mathbf V}

\def\H{\mathcal{H}}
\def\hausd{\mathcal H^2}
\def\spt{\operatorname{spt}}
\def\BV{\textit{BV}}
\def\dive{\operatorname{div}}
\def\loc{\operatorname{loc}}
\def\Det{\operatorname{Det}}
\def\dist{\operatorname{dist}}

\def\eqn#1$$#2$${\begin{equation}\label#1#2\end{equation}}

\long\def\colred#1\endred{{\color{red}#1}}
\long\def\colgreen#1\endgreen{{\color{green}#1}}
\long\def\colmagenta#1\endmagenta{{\color{magenta}#1}}
\long\def\colblue#1\endblue{{\color{blue}#1}}

\newtoks\by
\newtoks\paper
\newtoks\book
\newtoks\jour
\newtoks\yr
\newtoks\pages
\newtoks\vol
\newtoks\publ
\def\ota{{\hbox\vol{???}}}
\def\cLear{\by=\ota\paper=\ota\book=\ota\jour=\ota\yr=\ota
  \pages=\ota\vol=\ota\publ=\ota}
\def\endpaper{\the\by, {\the\paper},
  \textit{\the\jour} \textbf{\the\vol} (\the\yr), \the\pages.\cLear}
\def\endbook{\the\by, \textit{\the\book}, \the\publ.\cLear}
\def\endprep{\the\by, \textit{\the\paper}, \the\jour.\cLear}
\def\endyearprep{\the\by, \textit{\the\paper}, \the\jour, (\the\yr).\cLear}
\def\name#1#2{#2 #1}
\def\nom{ \rm no. }
\def\et{ and}

\newtheorem{thm}{Theorem}

\newtheorem{prop}[thm]{Proposition}
\newtheorem{lemma}[thm]{Lemma}
\theoremstyle{definition}
\newtheorem{defin}[thm]{Definition}
\newtheorem{rmrk}[thm]{Remark}

\title{On $\BV$ homeomorphisms}

\author[L. D'Onofrio]{Luigi D'Onofrio}
\address{Dipartimento di Scienze e Tecnologie, Universit\`a degli Studi di Napoli ``Parthe\-nope", Centro Direzionale Isola C4, 80100 Napoli, Italy}
\email{donofrio@uniparthenope.it}

\author[J. Mal\'y]{Jan Mal\'y}
\address{Department of Mathematical Analysis, Faculty of Mathematics and Physics, \break Charles University, So\-ko\-lovsk\'a 83,  Prague 8, 186\,75 Czech Republic}
\email{maly@karlin.mff.cuni.cz}

\author[C. Sbordone]{Carlo Sbordone}
\address{Dipartimento di Matematica e Applicazioni ``R. Caccioppoli",
Universit\`a degli Studi di Napoli  ``Federico II",
Via Cintia, 80126 Napoli, Italy}
\email{sbordone@unina.it}

\author[R. Schiattarella]{Roberta Schiattarella}
\address{Dipartimento di Matematica e Applicazioni ``R. Caccioppoli",
Universit\`a degli Studi di Napoli  ``Federico II",
Via Cintia, 80126 Napoli, Italy}
\email{roberta.schiattarella@unina.it}

\thanks{J.M.\ has been supported by the grant GA\,\v{C}R 
P201/18-07996S of the Czech Science Foundation.
 L. D., C. S. and R. S. are members of the Gruppo Nazionale per l'Analisi Matematica, la Probabilit\'a e le loro Applicazioni (GNAMPA) of the Istituto Nazionale di Alta Matematica (INdAM). The research of R.S. has been supported by the grant PRIN Project 2017JFFHSH.}

\begin{document}
\maketitle

\begin{abstract}
    We  obtain the rectifiability of the graph of a bounded variation homeomorphism $f$ in the plane and relations
    between gradients of $f$ and its inverse.     
    Further, we show an example of a bounded variation homeomorphism $f$ in the plane which satisfies the $(N)$ and $(N^{-1})$ properties and strict positivity of Jacobian of  both itself and its inverse, but neither $f$ nor $f^{-1}$ is Sobolev.    
    \end{abstract}

\section{Introduction}

Let $\Omega\subset\rn$ be an open set and $f\colon\Omega\to\rn$
be a Sobolev or $\BV$ homeomorphism. Then it is natural to ask
under what condition the inverse is also Sobolev or $\BV$.
For planar Sobolev mapping it has been solved 
by Hencl and Koskela  \cite{HK}. Then Hencl, Koskela and Onninen
\cite{HKO} proved that the inverse of a planar $\BV$ homeomorphism
is also $\BV$. 
For first results in the spatial case see 
Hencl, Koskela and Mal\'y \cite{HKM},
Onninen \cite{O} and Hencl, Koskela and Onninen  \cite{HKO}.
The $\BV$-regularity of the inverse of $W^{1,n-1}$ mappings has been
obtained by Cs\"ornyei, Hencl and Mal\'y \cite{CHM} and even a more precise
result for $n=3$ by Hencl, Kauranen and Luisto \cite{HKaL}.

The above mentioned result when the inverse of a Sobolev mapping is also Sobolev
were accompanied by the formula for the 
the gradient of the inverse; in fact it is the same as
in the smooth case. The corresponding result for $\BV$ mappings
is more difficult
even in the planar case.
For $f=(f_1,f_2)\in\BV(\Omega,\er^2)$, the identity
$$
|D_1f^{-1}|(f(U))=|Df_1|(U),\qquad U\subset\Omega\text{ open}
$$
has been shown by D'Onofrio and Schiattarella  \cite{DS}.
We prove the full characterization of $Df^{-1}$:

\begin{thm}\label{t:identity}
Let $\Omega\subset\er^2$ be a bounded open set
and $f\colon\Omega\to\er^2$ be a bounded sense preserving $\BV$-homeomorphism.
Let $U\subset\Omega$ be a Borel set. 
Then \\
{\rm(a)}
$
Df^{-1}(f(U))=\adj Df(U),
$\\
{\rm(b)}
$
\Det Df(U)=|f(U)|.
$
\end{thm}

Here
$$
\adj Df=\begin{pmatrix}
D_2f_2, &-D_2f_1\\
-D_1f_2, & D_1f_1
\end{pmatrix}\;,
$$
note that this is a measure.
The part (a) has been proved first 
by Quittnerov\'a in \cite{Qui}.
She also obtained the corresponding result in the spatial case
for $W^{1,n-1}$-homeomorphisms.
Unfortunately, her thesis has been never published.

Our proof is entirely different and shorter. While our project
has been in preparation, Hencl, Kauranen and Mal\'y \cite{HKaM}
proved a similar formula to (a) in dimension $3$ under the assumption
that both $f$ and $f^{-1}$ are $\BV$. They use a distributional 
adjugate. 

The expression $\Det Df$ appearing in (b) is the distributional Jacobian
introduced by Ball \cite{Ball}, see Definition \ref{d:distrjac}.
Also the part (b) follows from the degree formula
proved in \cite{HKaL} and \cite{HKaM}. We present our 
original proof which is simpler as we consider only the homeomorphic
case.

Whereas the main novelty of Theorem \ref{t:identity} is in new proofs,
our main result on rectifiability is entirely new.
The rectifiabilility of the graph has been known only for 
scalar $\BV$ functions (Federer \cite[4.5.9]{Fe}, see also Giaquinta, Modica and Sou\v{c}ek \cite[Section 4.1.5]{GMS}). For vector valued function it is available only due
to the assumption that our function is a homeomorphism.
The spatial case is open.

Graph of a Sobolev mapping $f$ on an $n$-dimensional
domain is rectifiable if the graph mapping satisfies
Luzin's $(N)$-condition, it means
\begin{equation}\label{Luzin}
E\subset \Omega,\ |E|=0 \;\implies \; \H^{n}(f(E))=0.
\end{equation}
This is commonly the case when $f$ itself satisfies the $(N)$-condition,
as the proof of $(N)$ for $f$ usually gives $(N)$ for the graph mapping.
For further discussion and typical results on the $(N)$-condition see
\cite{MMi}, \cite{MM}, \cite{KKM}, \cite{MSZ}. In our situation of a planar
homeomorphism, previously known results give rectifiability of the 
graph if $f\in W^{1,2}(\Omega)$ as follows from the result 
by Reshetnyak \cite{R1}.

To describe our rectifiability result,
we need to introduce
some notation. The symbol $\Gamma$ stands for the mapping
$x\mapsto (x,f(x))$, so that $\Gamma(\Omega)$ is the full
graph of $f$.
We consider a measure $\mu$ supported on $\Gamma(\Omega)$
with values in $\Lambda^2(\er^4)$; we write in coordinates
$$
\mu=\mu_{12}\,\e_1\wedge\e_2+
\mu^{12}\,\e^1\wedge\e^2+\sum_{i,j=1}^2\mu_i^j\,\e_i\wedge \e^j,
$$
where $(\e_1,\e_2)$ is the canonical basis of 
the domain space and $(\e^1,\e^2)$ is the canonical basis of 
the target space.
This means that given a smooth differential form with compact support
in $\Omega\times f(\Omega)$,
$$
\boldsymbol\omega=\omega_{12}\,dx_1\,dx_2+\omega^{12}\,dy_1\,dy_2
+\sum_{i,j=1}^2\omega_i^j\,dx_i\,dy_j,
$$
we have
$$
\langle \mu,\boldsymbol\omega\rangle =
\int_{\Gamma(\Omega)} \omega_{12}\,d\mu_{12}+\omega^{12}\,d\mu^{12}
+\sum_{i,j=1}^2\omega_i^j\,d\mu_i^j.
$$

The measure $\mu$ is defined as push forward of measures
on $\Omega$ through the graph mapping.
Namely, if 
$E\subset\Omega$ is a Borel set  then
\eqn{howmu}
$$
\aligned
\mu_{12}(\Gamma(E))&=|E|,\\
\mu_1^j(\Gamma(E))&=D_2f_j(E),\\
\mu_2^j(\Gamma(E))&=-D_1f_j(E),\\
\mu^{12}(\Gamma(E))&=|f(E)|=\Det Df(E).
\endaligned
$$
If $M\subset\er^4$ is a $2$-rectifiable set (see Definition \ref{d:rect}),
then it admits an orientation, see Definition \ref{d:orient}.
The structure of an orientation on $M$ can be alternatively 
expressed by
an integer multiplicity $2$-rectifiable current
of multiplicity $1$ on $M$, which is a measure on $M$
with values in $\Lambda_2(\er^4)$. 
For the definition of an integer multiplicity rectifiable current see
Definition \ref{d:curr}. In our situation we want
the orientation to coincide with the topological orientation;
this is expressed by the property that the current is 
boundaryless, see Definition \ref{d:bdless}.

\begin{thm}\label{t:main}
Let $\Omega\subset\er^2$ be a bounded open set
and $f\colon\Omega\to\er^2$ be a sense preserving bounded $\BV$-homeomorphism.
Then the graph $\Gamma(\Omega)$ of $f$ is a $2$-rectifiable 
set and $\mu$ is a boundaryless rectifiable current of 
multiplicity $1$ on $\Gamma(\Omega)$.
\end{thm}

Our second main result is a surprising example.

Let $\Omega\subset\er^2$ be an open set.
For a Sobolev homeomorphism $f\in W^{1,1} (\Omega;\er^2)$,
Luzin's $(N)$-condition 
is a useful property, as
it implies the \emph{area formula} for Borel sets $E\subset \Omega$,
\begin{equation}\label{area}
\int_E \, J_{f}(x)\, dx = |f(E)|
\end{equation}
and the change of variables can be used as for smooth transformations.
Moreover, the $(N)$-condition implies that the pointwise Jacobian determinant $J_f(x)= \det \Nabla f(x)$ coincides with the distributional Jacobian, $\Det \Nabla f(x)$, see \cite{DHMS}.
The formula
\begin{equation}\label{deteqDet}
\text{ det } \Nabla f(x)= \text{ Det } \Nabla f(x)
\end{equation}
should be interpreted that the distributional Jacobian is an absolutely 
continuous measure and the pointwise Jacobian acts as its density.
For the definition of the distributional Jacobian see Definition
\ref{d:distrjac}.

The positivity of $J_f$,
\begin{equation}\label{Jacobianposi}
J_f(x)>0 \quad \text{for a.\,e. } x \in \Omega,
\end{equation}
is also very useful because it guarantees that also $f^{-1}$ is Sobolev map (we say that $f$ is bi--Sobolev, see \cite{HMPS}) and the usual formula for the gradient of
the inverse holds.
Notice that the condition $J_f>0$ a.e. is equivalent to the $(N)$-condition for $f^{-1}$ (\cite{HKbook}).

Hence, if a Sobolev homeomorphism $f$ and its inverse $f^{-1}$ satisfy the $(N)$-condition, then $J_f>0$ a.e., $J_{f^{-1}}>0$ a.e.,
$f$ is bi--Sobolev, and \eqref{area}, \eqref{deteqDet} hold true for $f$ and $f^{-1}$.

On the contrary, for a planar $\text{\BV}$--homeomorphism $f$, whose inverse is automatically in $\BV_\text{loc}$ (see \cite{HKO}, \cite{CHM}), the assumption that $f$ and $f^{-1}$ verify the $(N)$-condition is not sufficient to gain $f^{-1}\in W^{1,1}_{\text{loc}}$. 

Therefore, our following example completes the picture what is possible:

\begin{thm}\label{t:example}
There exists a $\BV$-homeomorphism $f\colon [-1,1]^2\to[-1,1]^2$
such that both $f$ and $f^{-1}$ satisfy the $(N)$-condition,
$J_f>0$ a.e., $J_{f^{-1}}>0$ a.e.,
but neither $f$ nor $f^{-1}$ are Sobolev.
\end{thm}

There is a close link between the positive result and the example.
Indeed, the most difficult part of the proof of rectifiability
consist in the analysis of the part of the graph which is seen as vertical
for both $f$ and $f^{-1}$. For illustration of this behavior
we can use the part of the graph of $f$ from the example over
the Cantor set $S$ (see \eqref{theS}).

\section{Preliminaries}

We denote the Lebesgue measure of a set $E\subset \Rz^n$ 
by $|E|$ and the $k$-dimensional Hausdorff measure of $E$
by $\hausk(E)$.

Let $\Omega\subset\rn$ be an open set.
If $\boldsymbol m$ is a finite signed Radon measure on $\Omega$ and $u$
is an $\boldsymbol m$-integrable function, we denote 
$$
\langle\boldsymbol m, u\rangle = \int_{\Omega}u\,d\boldsymbol m;
$$
it is an extension of 
the duality between $\boldsymbol m\in C_0(\Omega)^*$ and $u\in C_0(\Omega)$.

Let $\Omega$ be a domain in $\Rz^n$. 
A function $u\in L^{1}(\Omega)$ is of bounded variation, $u\in \BV(\Omega)$, 
if the distributional partial derivatives of $u$ are measures with finite total variation in $\Omega$: there exist Radon signed measures $D_i u$ in $\Omega$ such that for $i=1,\dots,n$, $|D_i u|(\Omega)<\infty$ and
$$
\langle D_iu,\,\ff\rangle=-\int_{\Omega} u D_i\ff\,dx,\qquad
\ff \in C^{1}_c(\Omega).
$$
The gradient of $u$ is then the vector-valued measure $Du=(D_1u,\dots,D_nu)$,
$|Du|$ stands for its total variation. We have
$$
\left| D u \right|(\Omega)= \sup \left\{\int_{\Omega} u\; \dive \varphi\,dx \colon \varphi \in C^1_0(\Omega, \Rz^n), \| \varphi  \|_{\infty} \le 1 \right\}<\infty.
$$
We say that $u$ is a Sobolev function if its distributional 
gradient can be represented as a locally integrable function,
this means, $u\in W_{\loc}^{1,1}(\Omega)$.
The Sobolev space $W^{1,1} (\Omega)$ is contained in $\BV(\Omega)$
as $L^1(\Omega)$ can be regarded as a closed subspace of $\C_0(\Omega)^*$.

We say that $f\in L^{1}(\Omega; \Rz^m)$ belongs to $\BV(\Omega; \Rz^m)$ if each component of $f$ is a function of bounded variation. 
The total variation of $Df$ is then computed as $|Df_1|+\dots+|Df_n|$.
Finally we say that $f\in \BV_{\loc}(\Omega; \Rz^m)$ if $f\in \BV(U; \Rz^m)$ for every open $U\subset\subset\Omega$. 
The space $\BV(\Omega, \Rz^m)$ is endowed with the norm
$$
\| f \|_{\BV}:= \int_{\Omega}|f(x)|\,dx+ |Df|(\Omega).
$$

Now, we recall some useful tools.

\begin{prop}[Morse--Sard theorem, \cit{Lemma 13.15}{Mbook}]\label{Sard}
Let $\Omega\subset\rn$ be an open set.
If $\eta \in C^{\infty}(\Rz^n)$ and $E=\left\{x\in \Rz^n: \Nabla \eta(x)=0   \right\}$, then $|\eta(E)|=0$. In particular, $\left\{ \eta=t\right\}= \left\{x\in \Rz^n: \eta(x)=t   \right\}$ is a $\C^{\infty}$ hypersurface in $\Rz^n$ for a.e.\ $t\in \Rz$.
\end{prop}

\begin{prop}[Luzin approximation in $BV$, \cit{Theorem 5.34}{AFP}]
\label{p:Luzin}
Let $\Omega\subset\rn$ be an open set and
$f\in \BV_{\loc}(\Omega,\er^m)$. Then there are Lipschitz functions
$f_j\colon\rn\to\er^m$ such that 
$$
\Big|\bigcap\{x\in\Omega\colon f_j(x)\ne f(x)\}\Big|=0.
$$
\end{prop}

\begin{prop}[Co-area formula, \cit{4.5.9}{Fe}]\label{p:coarea}
Let $\Omega\subset\rn$ be an open set, $u\in BV(\Omega)$ be a continuous function
and
$\eta$ be a Borel function on $\Omega$, $\eta\ge 0$.
Then
\eqn{coarea}
$$
\langle |Du|,\eta\rangle =\int_{-\infty}^{\infty}
\Bigl(\int_{\{u=t\}}\eta\,d\haus\Bigr)\,dt.
$$
\end{prop}

Given a smooth map $f$ from $\Omega\subset\rn$ into $\rn$ 
and $U\subset\subset\Omega$ we can define the topological degree as
$$\deg(f,U,y_0)=\sum_{\{x\in U: f(x)=y_0\}} \sgn(J_f(x))$$
if $J_f(x)\neq 0$ for each $x\in f^{-1}(y_0)$.
This definition can be extended to an arbitrary continuous mapping and each point
of $\rn\setminus f(\partial U)$, 
see e.g.\ \cite{FG} for the definition and properties of the degree.

A continuous mapping
$f:\Omega\to\rn$ is called \emph{sense-preserving} if
$$\deg(f,U,y_0)>0$$ for all domains $U\subset\subset\Omega$ and
all $y_0\in f(U)\setminus f(\partial U)$. Similarly we call $f$ \emph{sense-reversing} if
$\deg(f,U,y_0)<0$ for all $U$ and $y_0$.
Let us recall that each homeomorphism on a domain is either sense-preserving or sense-reversing, see
\cite[II.2.4., Theorem 3]{RR}.

\begin{prop}[Degree formula, \cite{FG}]\label{p:degree}
Let $\Omega\subset\rn$ be an open set and 
$f\colon\Omega\to\rn$ be a $\C^1$-mapping. Let $U\subset\subset\Omega$
be an open set. Then
$$
\int_{U}\eta(f(x))\,J_f(x)\,dx=\int_{\rn}\eta(y)\,\deg(f,U,y)\,dt
$$
for each continuous function $\eta$ with compact support
in $\rn\setminus f(\partial U)$.
\end{prop}

\begin{defin}\label{d:rect}
We say that a set $M\subset\rn$ is \textit{countably $k$-rectifiable} 
if $M$ can be written as $N\cup\bigcup_{j=1}^{\infty}\psi_j(E_j)$,
where $E_j\subset\er^k$ are measurable, $\psi_j\colon E_j\to \rn$
are Lipschitz, and $\hausk(N)=0$.
We say that  $M$ is \textit{$k$-rectifiable} if it is 
countably $k$-rectifiable and $\hausk(M)<\infty$.
\end{defin}

\begin{rmrk}\label{r:single}
We can reduce $(\psi_j)_j$ to a single locally Lipschitz mapping $\psi:E\to \rn$,
where $E\subset\rn$ is measurable. Indeed, it is enough to split the 
sets $E_j$ into small pieces and reorganize them by translations
to be mutually distant. Then we have
$$
M=\psi(E)\cup N
$$
where $\hausk(N)=0$ and we call $\psi$ to be a \textit{parametrization}
of $M$.
The set $E$ can be very wild and scattered even if $M$ is topologically nice.
\end{rmrk}

\begin{defin}
We denote the space of $k$-vectors on $\rn$ by
$\Lambda_k(\rn)$ and the space of $k$-covectors on $\rn$ by $\Lambda^k(\rn)$.
A differential $k$-form on an open set $W\subset\rn$ is a mapping 
$\boldsymbol\omega\colon W\to \Lambda^k(\rn)$.
For more information we refer to \cite{Fe}.
\end{defin}

\begin{defin} If $\sigma$ is a Radon measure on $\rn$,
$x\in\rn$, and $r>0$, we denote by $\sigma_{x,r}$ the measure
which acts on any function $\ff\in C_c(\rn)$ as 
\eqn{blowup}
$$
\langle \sigma_{x,r},\ff\rangle = \int_{\rn} \ff\Bigl(x+\frac{y-x}{r}\Bigr)\,d\sigma(y).
$$
\end{defin}

\begin{prop}[\cit{Thm.\ 11.8}{Simon}]\label{p:simon}
 Let $M\subset\rn$ be a set and 
$\sigma$ be a Radon measure on $\rn$ such that $\sigma(\rn\setminus M)=0$.
Suppose that for each $x\in\rn$ there exists 
a $k$-dimensional linear subspace $\ve\subset\rn$ such that
$$
\lim_{r\to0^+}\int_{\rn}\ff\,d\sigma_{x,r}=\int_{\ve}\ff\,d\hausk
$$
for any $\ff\in \C_c(\rn)$.
Then $M$ is countably $k$-rectifiable and $\sigma=\hausk\lfloor M$.
\end{prop}

\begin{defin}\label{d:orient}  
Let $M\subset\rn$ be a countably $k$-rectifiable set.
The linear subspace $\ve$ from Proposition \ref{p:simon}
is called an \textit{approximate tangent space} to $M$ at $x\in M$.
By \cite[Thm.\ 11.6]{Simon}, an approximate tangent plane
exists at  $\hausk$-a.e.\ $x\in M$; by \cite[Rmrk.\ 11.5]{Simon}
it is uniquely determined up to a $\hausk$-null set.

We can consider an \textit{orientation} on $M$, by this we mean
that the approximate tangent space is oriented at $\hausk$-a.e.\ $x\in M$,
so that the
$k$-vector field $\xi$ on $M$,
where $\xi(x)$ is obtained as
the exterior product
of an ortonormal basis of the approximate tangent space to $M$ at $x$,
is $\hausk$-measurable.
We write the structure of orientation on $M$ as $(M,\xi)$
and call $(M,\xi)$ an \textit{oriented $k$-rectifiable set.}
A locally Lipschitz mapping $\psi\colon E\to \rn$, where $E\subset\er^k$,
is measurable,
is called a \textit{positive parametrization} of $(M,\xi)$ 
if $\psi$ is a parametrization of  $M$ and $(D_1\psi_1(t),\dots,D_k\psi(t))$
is a positive basis of the approximate tangent space to $M$ at $\psi(t)$
for a.e.\ $t\in E$. Here, if $E$ is not a neighborhood of $t$,
we use derivatives of a Lipschitz extension of $\psi$, they are
independent of the choice of the extension a.e.\ in $E$.
\end{defin}

\begin{rmrk} 
Similarly to the argument of Remark \ref{r:single}, it can be shown that
a positive parametrization of an oriented $k$-rectifiable set always exists.
\end{rmrk}

\begin{defin}\label{d:curr} 
Let $\nu$ be a measure with values in $\Lambda_k(\rn)$. We say 
that $\nu$ is an \textit{integer multiplicity $k$-rectifiable current}
if it acts on differential $k$-forms as 
$$
\langle \nu,\boldsymbol\omega\rangle = \int_{M}\theta(x)\langle\boldsymbol\omega(x),\xi(x)\rangle\,d\hausk(x),
$$
where $(M,\xi)$ is an oriented countably $k$-rectifiable set
and $\theta\colon M\to\mathbb N$
is a $\hausk$-measurable function (multiplicity).
We write $\nu=\nu(M,\xi,\theta)$. 
We call $\nu$ a \textit{simple $k$-rectifiable current},
and simplify to $\nu(M,\xi)$,
if $\theta\equiv1$ (this is our situation throughout the paper).
\end{defin}

If we have a structure of a simple $k$-rectifiable current
on $M$ where $M$ is a connected topological manifold,
we may want to express that the orientation is compatible
with the topological orientation of $M$. The right 
interepretation of this is the property of being boundaryless.

\begin{defin}\label{d:bdless} Let $W\subset\rn$ be an open set
and $\nu=\nu(M,\xi)$ is an integer multiplicity $k$-rectifiable current
with $M\subset W$. We say that $\nu$ is \textit{boundaryless in $W$}
if $\langle\nu, d\boldsymbol\omega\rangle=0$ for each 
smooth differentiable $(k{-}1)$-form $\boldsymbol\omega$ with compact support 
in $W$. 
\end{defin}

\begin{defin}\label{d:distrjac}
Let $\Omega\subset\er^2$ be an open set and 
$f\colon\Omega\to\er^2$ be a continuous $BV$ mapping.
Then we define the \textit{distributional Jacobian}
$\Det Df$ as 
$$
\langle Df,\ff\rangle = \lim_{j\to \infty}
\int_{\Omega}\ff(x)\,J_{f_j}(x)\,dx,
$$
where $f_j$ are standard convolution approximations of $f$.
The limit has a good sense as integration by parts leads to
a valid duality between measures and continuous functions, 
see Remark \ref{r:distrjac}.

If $u\in \BV(\Omega)$ is continuous and $\ff\colon\Omega\to\er$
is smooth, we need neither approximation nor distributional differentiation
to give sense to the determinant  $\det (D\ff,Df_2)$. Therefore we 
denote it by lowercase ``det'' although it is a measure. Namely,
$$
\langle\det (D\ff,Du),\eta\rangle = 
\langle D_2u,\eta\,D_1\ff\rangle -\langle D_1u,\eta\,D_2\ff\rangle,\qquad
\eta\in \C_c(\Omega).
$$

\end{defin}

\begin{rmrk}\label{r:distrjac}
If $\Phi\colon\er^2\to\er^2$ is a $\C^1$-mapping satisfying $\dive\Phi\equiv 1$
and $\ff\in \C_c^{\infty}(\Omega)$, then 
$$
\langle \Det Df,\ff\rangle = 
\langle \det (D\ff,Df_1),\,\Phi_2\circ f\rangle
-\langle \det (D\ff,Df_2),\,\Phi_1\circ f\rangle
,\qquad \ff\in \C_c^{\infty}(\Omega).
$$
In particular,
$$
\langle \Det Df,\ff\rangle = -\langle \det(D\ff, Df_2), \,f_1 \rangle
,\qquad \ff\in \C_c^{\infty}(\Omega),
$$
which is frequently used as a definition.  
\end{rmrk}


\section{Gradient of the inverse}

In this section we prove the part (a) of Theorem \ref{t:identity}.
We proceed similarly as in \cite{HKM}.
Recall that $f\colon\Omega\to\er^2$ is a sense preserving 
$BV$ homeomorphism.
Without loss of generality, we can restrict our attention to one coordinate, 
namely we want to prove that 
\eqn{special}
$$
D_2(f^{-1})_2(f(U))=D_1f_1(U)
$$
for each open set $U\subset\er^2$.
To this end, we use an auxiliary mapping
$$
g(x)=(f_1(x),x_2).
$$

\begin{lemma}\label{l:withsmooth}
Let  $U\subset\subset\Omega$ be an open set. 
Then
$$
\int_{\er^2}\eta(z)\,\deg(g,U,z)\,dz=
\langle D_1f_1, \eta\circ g\rangle
$$
for each $\eta\in\C_c^{\infty}(\er^2)$ with
$\spt\eta\cap g(\partial U)=\emptyset$
and
$$
\int_{\er^2\setminus g(\partial U)}|\deg(g,U,z)|\,dz
\le |D_1f_1|(U).
$$
\end{lemma}

\begin{proof}
Assume first that $g$ is smooth. 
Using the degree formula (Proposition \ref{p:degree})
we compute
$$
\aligned
\int_{\er^2}\eta(z)\,\deg(g,U,z)\,dz
&=\int_{U}\eta(g(x))\,J_g(x)\,dx
\\
&=\int_{U}\eta(g(x))\,D_1f_1(x)\,dx.
\endaligned
$$
In the general case we approximate $g$ by convolution and 
pass to the limit.
Now, we use the fact that
$$
\int_{\er^2\setminus g(\partial U)}\!\!\!\!\!\!|\deg(g,U,z)|\,dz
=\sup\Bigl\{\int_{\er^2}\!\!\eta(z)\,\deg(g,U,z)\,dz\colon 
\eta\in\C_c^{\infty}(\er^2{\setminus} g(\partial U)),\,|\eta|\le 1\Bigr\}.
$$
\end{proof}

\begin{lemma}\label{l:degg}
Let  $U\subset\subset\Omega$ be an open set
and $|g(\partial U)|=0$.
Then
$$
\int_{\er^2}\deg (g,U,z)\,dz= D_1f_1(U).
$$
\end{lemma}
\begin{proof}
Let $\eta_j\in\C_c^{\infty}(\er^2)$ be smooth functions satisfying
$\spt\eta_j\cap g(\partial U)=\emptyset$ and 
$\eta_j\nearrow 1$ on $g(U)\setminus g(\partial U)$. 
By Lemma \ref{l:withsmooth},
$$
\aligned
\int_{\er^2}\deg (g,U,z)\,dz&= 
\lim_{j\to\infty}\int_{\er^2}\eta_j(z)\deg (g,U,z)\,dz
=\lim_{j\to\infty}\langle D_1f_1,\;\eta_j\circ g\rangle 
\\&=D_1f_1 (U).
\endaligned
$$
The passage to the limit is justified as
$\deg (g,U,\cdot)$ is integrable by Lemma \ref{l:withsmooth}
and $D_1f_1$ is a finite measure.
\end{proof}

\begin{proof}[Proof of Theorem \ref{t:identity}(a)]
Let $U\subset\subset\Omega$ be an open set.
We find a monotone exhaustion
$U=\bigcup_{j}U_j$ where, for each $j=1,2,\dots$,
$U_j\subset\subset U$ is an open set with a smooth boundary
and $g$ has a $BV$ trace on $\partial U_j$.
Then
$|g(\partial U_j)|=0$.
By Lemma \ref{l:degg} we have
$$
\adj_{22}Df(U_j)=D_1f_1(U_j)=\int_{\er^2}\deg(g,U_j,z)\,dz.
$$
Now, we apply Lemma \ref{l:degg} to the inverse, but instead 
of $g(x)=(f_1(x),x_2)$ we consider $h(y)=(y_1,(f^{-1})_2(y))$.
We obtain 
$$
D_2(f^{-1})_2(f(U_j))=\int_{\er^2}\deg(h,f(U_j),z)\,dz.
$$
However, by the formula on composition of degree,
we have
$$
\deg(g,U_j,z)=\deg(h,f(U_j),z),\qquad z\in \er^2\setminus g(\partial U_j),
$$
as $g=h\circ f$ and $f$ is a sense preserving homeomorphism.
Therefore
$$
D_2(f^{-1})_2(f(U_j))=\adj_{22}Df(U_j),\qquad j=1,2,\dots,
$$
and passing to the limit we obtain
$$
D_2(f^{-1})_2(f(U))=\adj_{22}Df(U).
$$
By a routine measure-theoretic approximation we obtain the formula
for each Borel set.
\end{proof}

\section{Distributional Jacobian}

Let $\Omega\subset\er^2$ be a bounded open set and $f=(f_1, f_2)\colon \Omega\to\er^2$ be a bounded sense-preserving $BV$-homeomorphism.

Let $G\subset\subset\Omega$ be an open set with a smooth boundary. Then  each connected component $T$
of the boundary of $G$ is a smooth Jordan curve which can be parametrized as $\gamma\colon [0,1]\to\er^2$, where $\gamma$ is smooth, regular,
and oriented so that $\det(\mathbf n^G(\gamma(t)),\gamma'(t))>0$, where $\mathbf n^G$ is the outward normal to $G$.
Given a continuous function $u$ and a $\BV$ function $v$ on $T$, we consider the Stieltjes integral $\int_{T}u\,dv$
(and thus also the Stieltjes integral over $\partial G$) which is defined via the parametrization $\gamma$ but does not depend on it.

\begin{lemma}
Let $u,v\colon\Omega\to\er$ be continuous $BV$ functions and $\eta\in \C_c^{\infty}(\Omega)$ be nonnegative. Then
\eqn{jacu}
$$
\langle\det\left(Dv,D\eta\right),u\rangle
=\int_{0}^{\infty}\Bigl(\int_{\{\eta=t\}}u\,dv\Bigr)\,dt
$$
and
\eqn{jacv}
$$
\langle\det\left(Du,D\eta\right),v\rangle
=\int_{0}^{\infty}\Bigl(\int_{\{\eta=t\}}v\,du\Bigr)\,dt.
$$
\end{lemma}

\begin{proof} 
By the Morse--Sard theorem (Proposition \ref{Sard}), there is a set $Z\subset\er$ of measure zero such that 
$\{\eta>t\}$ has a smooth boundary for all $t\in (0,+\infty)\setminus Z$.
Let $t>0$ be such that $\{\eta=t\}$ has a smooth boundary.
Then we have
$$
D\eta(x)= -\mathbf n^{\{\eta>t\}}(x)\,|D\eta(x)|,\qquad x\in\{\eta=t\}.
$$
To prove \eqref{jacu} and \eqref{jacv},
let us first assume that $v$ is smooth.
Applying the co-area formula \eqref{coarea} we obtain
$$
\aligned
\langle\det\left(Dv,D\eta\right),u\rangle&=
\int_{\Omega}u\det\left(Dv,D\eta\right)\,dx=
\\&= -\int_{0}^{\infty}\left(\int_{\{\eta=t\}}u \det(D v,\mathbf n^{\{\eta>t\}})\,d\hauso\right)\,dt
\\&=\int_{0}^{\infty}\Bigl(\int_{\{\eta=t\}}u\,dv\Bigr)\,dt.
\endaligned
$$
This proves \eqref{jacu} in the case of smooth $v$.
If $u$ is also smooth, we observe that
\eqn{byparts}
$$
\langle\det\left(Dv,D\eta\right),u\rangle=-\langle\det\left(Du,D\eta\right),v\rangle
$$
and
\eqn{smallbyparts}
$$
\int_{\{\eta=t\}}u\,dv=-\int_{\{\eta=t\}}v\,du
$$
holds for $t>0$, $t\notin Z$.
A simple approximation argument shows that \eqref{byparts}  holds
in general for $u,v$ which are both continuous and $\BV$ on $\Omega$
and \eqref{smallbyparts} holds for $t>0$ such that $t\notin Z$ and 
$u$ and $v$ are continuous $\BV$ on $\{\eta=t\}$.
Then, assuming still that $v$ is smooth, using \eqref{byparts}
and \eqref{smallbyparts}
we obtain
$$
\aligned
\langle\det\left(Du,D\eta\right),v\rangle&=
\langle\det\left(Dv,D\eta\right),u\rangle
\\&=\int_{0}^{\infty}\Bigl(\int_{\{\eta=t\}}u\,dv\Bigr)\,dt
=\int_{0}^{\infty}\Bigl(\int_{\{\eta=t\}}v\,du\Bigr)\,dt.
\endaligned
$$
This proves \eqref{jacv} in the case of smooth $v$.
Now, in the general case we use convolution approximation
of $v$ by smooth functions and obtain \eqref{jacv} as well.
For the passages to limits on the right hand side we 
use the dominated convergence theorem and the fact that integral over $t$ of the variation of
$v$ on $\{\eta=t\}$ is finite.
Now, interchanging the roles of $u$ and $v$ we obtain 
\eqref{jacu} in the general case.

\end{proof}

\begin{thm}\label{t:det}
Let $f\colon \Omega\to\er^2$ be a bounded sense-preserving $BV$ homeomorphism. If $E\subset\Omega$
is a Borel set, then
$\Det Df (E) = |f(E)|$.
\end{thm}

\begin{proof}
We compare the distribution $\Det Df$ and the measure $\rho\colon E\mapsto |f(E)|$.
Choose a nonnegative test function $\eta\in \C_c^{\infty}(\Omega)$. By Lemma \ref{coarea}, Green's theorem and integration by means of the distribution function
we have
$$
\aligned
\langle \Det Df,\eta\rangle&=
-\langle \det (D\eta,Df_2),\,f_1\rangle
=\int_{0}^{\infty}\Bigl(\int_{\{\eta=t\}}f_1\,df_2\Bigr)\,dt
\\&=\int_{0}^{\infty}\Bigl(\int_{\{\eta\circ f^{-1}=t\}}y_1\,dy_2\Bigr)\,dt
\\&=\int_{0}^{\infty}\Bigl|\{\eta\circ f^{-1}>t\}\Bigr|\,dt=
\int_{f(\Omega)}\eta(f^{-1}(y))\,dy=\int_{\Omega}\eta\,d\rho.
\endaligned
$$
Therefore the measures $\Det Df$ and $\rho$ coincide.
\end{proof}

\section{Fundamental estimate}

\begin{thm}\label{t:key} Let $f\colon \Omega\to\er^2$ be a sense-preserving $BV$ homeomorphism and $(x,y)\in \Gamma(\Omega)$. Then there exists $r_0>0$ such that for each $r\in (0,r_0)$ we 
have
$$
r^2\le C|\mu|(B(x,r)\times B(y,r)).
$$
\end{thm}

\begin{proof} Choose $x_0\in\Omega$ and denote $y_0=f(x_0)$.
Set 
$$
r_0=\min\{\dist(x_0,\Omega^c),\dist(y_0,f(\Omega)^c)\}
$$ 
and
choose $r\in (0,r_0)$.
If 
$f(B(x_0,\frac12r))\subset B(y_0, r)$, then $$\mu_{12}(B(x_0,r)\times B(y_0,r))\ge |B(x_0,\frac12 r)|.$$
If 
$f(B(x_0,r))\supset B(y_0, \frac12r)$, then $$\mu^{12}(B(x_0,r)\times B(y_0,r))\ge |B(y_0,\frac12 r)|.$$
In the remaining case, for each $t\in [\frac12 r, r]$,
$K_t:=f^{-1}(\partial B(y_0,t))$
is a continuum which intersects both $\partial B(x_0,\frac12 r)$
and $\partial B(x_0,r)$ and thus its diameter is at least $\frac12 r$.
Using the co-area formula (Proposition \ref{p:coarea}), we obtain
$$
\aligned
|\mu|(B(x_0,r)\times B(y_0,r))&\ge |Df|(B(x_0,r))\ge |D|f-y_0||((B(x_0,r))
\\&\ge \int_{r/2}^r\hauso(\{|f-y_0|=t\})\ge \frac14 r^2.
\endaligned
$$
\end{proof}

\begin{thm}\label{t:keyc} 
Let $f\colon \Omega\to\er^2$ be a sense-preserving $BV$ homeomorphism and $E\subset \Omega$ be a Borel set. Then 
$$
\hausd(\Gamma(E))\le C|\mu|(\Gamma(E)).
$$ 
In particular, $\hausd(\Gamma(\Omega))<\infty$.
\end{thm}

\begin{proof}
It follows from Theorem \ref{t:key} by the Vitali type covering argument. 
\end{proof}

\section{Rectifiability}

\begin{thm}\label{t:rect}
Let $f\colon \Omega\to\er^2$ be a sense-preserving $BV$ homeomorphism
and $\Gamma$ be its graph.
Let $\tau(x,f(x))$ be the Radon-Nikodym derivative
of $D\Gamma$ with respect to $|D\Gamma|$ 
at $x$ and $\tau_1(x,f(x))$, $\tau_2(x,f(x))$ be
the columns of this matrix.
Then $\Gamma(\Omega)$ is $2$-rectifiable and $(\tau_1,\tau_2)$ is a basis
of the approximate tangent plane to $\Gamma(\Omega)$ $\hausd$-a.e.
\end{thm}

\begin{proof}
We use the blow-up idea of De Giorgi's proof of rectifiability 
of the reduced boundary of a set of finite perimeter,
see e.g.\ \cite[Thm. 3.59]{AFP}, however, our situation is more
complicated.
Let $\mu$ be as in \eqref{howmu} and $\sigma=|\mu|$.
Let $\kappa=\frac{d\mu}{d\sigma}$.
Let $M_0$ be the set of all Lebesgue points of $\kappa$
with respect to $\sigma$.
By the Lebesgue-Besicovitch differentiation theorem, 
$\sigma(\Gamma\setminus M_0)=0$. We split $M_0$ as follows:
$$
\aligned
M_{12}&=\{x\in M_0\colon \kappa_{12}(x)>0\},\\
M^{12}&=\{x\in M_0\colon \kappa^{12}(x)>0\},\\
M^*&=\{x\in M_0\colon \kappa_{12}(x)=\kappa^{12}(x)=0\}.
\endaligned
$$
We want to prove $2$-rectifiability of these sets.

By the Luzin approximation (Proposition \ref{p:Luzin}),
there is a set $E\subset\Omega$ such that $|E|=0$
and $\Gamma(\Omega\setminus E)$ 
can be covered by countably many Lipschitz graphs.
Since $\mu_{12}(\Gamma(E))=0$, we have $\kappa_{12}=0$ 
$\sigma$-a.e.\ in $\Gamma(E)$. Therefore 
$\sigma$-almost all of $M_{12}$ is $2$-rectifiable.
By Theorem \ref{t:keyc}, the exceptional set
has $\hausd$-measure zero and thus $M_{12}$ is 
$2$-rectifiable.
We have $D_1\Gamma(x)=(1,0,D_1f_1(x),D_1f_2(x))$ and 
$D_2\Gamma(x)=(0,1,D_2f_1(x),D_2f_2(x))$ a.e. (here the derivatives have the 
meaning of pointwise representatives of the absolutely continuous part),
so that $\kappa$ is a positive multiple of 
$\tau_1\wedge \tau_2$ $\hausd$-a.e.\ on $M_{12}$.

Passing to the inverse using Theorem \ref{t:identity}
we see that also $M^{12}$ is $2$-rectifiable and
 $(\tau_1,\tau_2)$ is a basis
of the approximate tangent plane to $\Gamma(\Omega)$ $\hausd$-a.e.\ 
on $M^{12}$.

We proceed to the set $M^*$. Due to its definition,
$\sigma$ is the push-forward of $|D\Gamma|$ but also of $|Df|$
through $\Gamma$ on
$M^*$, as $\mu_{12}$ and $\mu^{12}$ vanish on $M^*$ and remaining
coordinates are coordinates of $D_jf_i$ upto their sign and 
arrangement. It follows that $\tau_1=(0,0,-\kappa_2^1,-\kappa_2^2)$
and  $\tau_2=(0,0,\kappa_1^1,\kappa_1^2)$ holds $\sigma$-a.e.\ on $M^*$.
We want to verify the assumptions of Proposition \ref{p:simon}. To this end,
observe that for $\sigma$-almost all $(x,y)\in M^*$ we have
\eqn{localis}
$$
\lim_{r\to0^+} \frac{\sigma (B((x,y),r)\setminus M^*)}{\sigma (B((x,y),r)}=0
$$
This is a consequence of the Lebesgue-Besicovitch differentiation 
theorem. Therefore it does not matter that $\sigma$ is not
carried by $M^*$.
By  Alberti' rank-one theorem \cite{Alb}, the Radon-Nikodym derivative
of $D_sf$ with respect to $|Df|$ is rank-one. It follows that the
Radon-Nikodym derivative of $(\mu_i^j)_{i,j=1}^2$ with respect 
to $\sigma$ is also rank-one.
Consequently,  $(\kappa_i^j)_{i,j=1}^2$ is rank-one $\sigma$-a.e.\ in $M^*$.
Pick $(x_0,y_0)\in M^*$ with this rank-one property and such that
$|\kappa(x_0,y_0)|=1$ and \eqref{localis} holds.
For simplicity let us assume that $x_0=y_0=0$.
By a rotation both in
domain and range we may assume that 
\eqn{rotation}
$$
\kappa_1^2(0,0)=\kappa_2^2(0,0)
=\kappa_2^1(0,0)=0,\quad
\kappa_1^1(0,0)=1.
$$
We want to prove that $\ve=\{(x,y)\in\er^2\times\er^2\colon x_2=y_2=0\}$
satisfies the assumptions of Proposition \ref{p:simon}. 
For $r>0$ small enough write
$$
f^r(x)=\frac{f(rx)}{r}.
$$
Consider functions $\ff_1$, $\ff_2\in \C_c^{\infty}(\er)$ and 
denote their indefinite integrals by $\Phi_1$ and $\Phi_2$,
respectively, normalized by $\Phi_i(0)=0$. 
Let $\eta\in \C_c^{\infty}(\er^2)$. Let $B$
be a ball containing the support of $(x,y)\mapsto\eta(x)\ff_1(y_1)\ff_2(y_2)$.

A change of variables yields
\eqn{cv}
$$
\int_{\Gamma(\Omega)}
\ff_1\Big(\frac {\bar y_1}{r}\Big)
\ff_2\Big(\frac {\bar y_2}{r}\Big)
\eta\Big(\frac{\bar x}{r}\Big)
d\mu_1^1(\bar x,\bar y)
=-r^2\langle
D_1f_1^r,\,(\ff_1\circ f^r_1)\;(\ff_2\circ f^r_2)\; 
\eta\rangle
\,,
$$
similarly for other coordinates of $\mu$.

We want to prove that 
$$
r^{-2}\sigma(B(0,r)\times [-r,r]^2)
$$
remains bounded as $r\to 0^+$.
Assume for a while that $\ff_1$, $\ff_2$ and $\eta$ have the shape of 
cut-off functions. For simplicity, assume that 
$\ff_1\in \C_c^{\infty}(\er)$ is even,  $0\le \ff_1\le 1$, 
$\ff_1=1$ on $[-1,1]$, $\ff_1=0$ outside $(-2,2)$, and
$\ff_2=\ff_1$, whereas $\eta(x)=\ff_1(|x|)$ on $\er^2$.

For the estimate, it is enough to consider $r$ such that 
$$
r^{-2}\sigma(B(0,r)\times [-r,r]^2)\ge (2r)^{-2}\sigma(B(0,2r)\times [-2r,2r]^2).
$$
Given $\ep>0$, for $r$ small enough we have
$$
r^{-2}\sigma(B(0,r)\times [-r,r]^2)
\le 2r^{-2}\mu_{1}^1(B(0,r)\times [-r,r]^2)
$$
and
$$
|\mu_1^2|(B(0,2r)\times [-2r,2r]^2)
\le \ep \sigma (B(0,2r)\times [-2r,2r]^2)
$$
by \eqref{rotation}. We estimate
$$
\aligned
r^{-2}\sigma&(B(0,r)\times [-r,r]^2)
\le 2r^{-2}\mu_{1}^1(B(0,r)\times [-r,r]^2)
\\&
\le -2\langle D_1f^r_1, \; (\ff_1\circ f^r_1)\;(\ff_2\circ f^r_2)\; 
\eta\rangle
\\&
=-2\langle D_1(\Phi_1\circ f^r_1), \;(\ff_2\circ f^r_2) \; \eta\rangle
\\&=2\langle D_1(\ff_2\circ f^r_2)), \; (\Phi_1\circ f^r_1)\; \eta\rangle
+2\int_{B} (\Phi_1\circ f^r_1)(\ff_2\circ f^r_2))\,  D_1 \eta\,dx
\\&\le C(r^{-2}|\mu_1^2|(B(0,2r)\times [-2r,2r]^2)+1)
\le C(\ep r^{-2}\sigma (B(0,2r)\times [-2r,2r]^2)+1)
\\&\le 
4C\ep r^{-2}\sigma (B(0,r)\times [-r,r]^2)+ C.
\endaligned
$$
Taking $\ep$ such that $8C\ep<1$ we conclude that 
\eqn{bdd}
$$
r^{-2}\sigma(B(0,r)\times [-r,r]^2)\le C.
$$

Consider $r_j\searrow 0$ and write $f^{(j)}=f^{r_j}$.
In view of \eqref{bdd}, \eqref{cv} and \eqref{rotation} we
observe that 
$$
\langle |Df_2^{(j)}| , \;(\ff_1\circ f^{(j)}_1)(\ff_2\circ f^{(j)}_2) \; \eta\rangle\to 0,
$$
$$
\langle |D_2f_1^{(j)}| , \;(\ff_1\circ f^{(j)}_1)(\ff_2\circ f^{(j)}_2) \; \eta\rangle\to 0.
$$
and
$$
\sup_j\langle |Df_1^{(j)}| , \;(\ff_1\circ f^{(j)}_1)(\ff_2\circ f^{(j)}_2) \; \eta\rangle<\infty.
$$
Hence $|D(\Phi_1\circ f_1^{(j)})|\to 0$ and $\sup_j|D(\Phi_1\circ f_1^{(j)})|<\infty$.
Passing if necessary to a subsequence,
the sequence $g_j:=-\Phi_1\circ f_1^{(j)}$ converges weakly* in $BV(B(0,1))$ and 
strongly in $L^1(B(0,1))$ (in view of compact embedding) to a limit function $g$. 
Since $\|D_2g_j\|\to 0$, the function $g$ depends only on the $x_1$-variable.
Let $R=\Phi_1(2)$. Since the measure of $\{x\colon |g_j(x)|<R\}$ 
is estimated 
by $Cr_j^{-2}\mu_{12}(B(0,2r_j)\times [-2r_j,2r_j]^2)\to 0$,
passing to the limit we  have $|g|=R$ a.e. 
In view of Theorem \ref{t:key}, $(\|Dg_j\|)_j$ is bounded away from $0$
and by \eqref{rotation} we see that the negative parts of $D_1g_j$ tend to $0$.
It follows that there is $c\in\er$ such that $g(x)=-R$ for $x_1<c$
and $g(x)=R$ for $x_1>0$. Assuming that $c>0$, we would get
that $g\equiv -R$ when replacing $r_j$ with $\frac12cr_j$.
This would contradict the estimate of Theorem \ref{t:key}.
Similarly we would exclude $c<0$. Hence $c=0$ and 
$$
g=
\begin{cases}
R,& x_1>0,\\
-R,& x_1<0.
\end{cases}
$$
Now, let $\ff_1$, $\ff_2$ and $\eta$ be general
(this means, smooth with compact support, but not necessarily the 
cut-off functions).
We may assume that the support of $\ff_i$ is contained in $[-1,1]$,
$i=1,2$, and the support of $\eta$ is contained in $B(0,1)$. 
Let, again, $g_j=-\Phi_1\circ f_1^{(j)}$. Then, argumenting as above,
$g_j\to g$ weakly* in $BV(B(0,1))$ and strongly in $L^1(B(0,1))$, where
$$
g=
\begin{cases}
\Phi_1(1),& x_1>0,\\
\Phi_1(-1),& x_1<0.
\end{cases}
$$
Then
$$
\aligned
&-\langle D_1f_1^{(j)} , \;(\ff_1\circ f^{(j)}_1)(\ff_2\circ f^{(j)}_2) 
\; \eta\rangle\to 
(\Phi_1(1)-\Phi_1(-1))\ff_2(0)\int_{\er}\eta(x_1,0)\,dx_1
\\&\quad=
\int_{\er^2}\ff_1(y_1)\ff_2(0)\; \eta(x_1,0)\,dx_1\,dy_1.
\endaligned
$$
Under the notation as in \eqref{blowup} we observe that 
$$
\int_{B(0,1)\times [-1,1]^2}\Psi(x,y)\,d(\mu_1^1)_{0,r_j}(x,y)\to
\int_{x_2=y_2=0}\Psi(x,y)\,d\hausd(x,y),
$$
if $\Psi(x,y)$ is of the form $\ff_1(y_1)\ff_2(y_2)\eta(x)$ with
$\ff_1$, $\ff_2$ and $\eta$ smooth with compact support.
However, linear combinations of such functions are dense
in $\C_0(\er^4)$ and the sequence $((\mu_1^1)_{0,r_j})_j$ is locally bounded,
it follows that $(\mu_1^1)_{0,r_j}\to\hausd\lfloor_{\{x_2=y_2=0\}}$
weakly* on any bounded ball in $\er^4$.
We have proved that from any sequence $r_j$ we can extract a subsequence
that the weak* convergence to the same limit occurs for it.
This is enough to conclude that the weak* convergence occurs as
$r\to 0^+$, since the weak* topology is metrizable on bounded
subsets of $\C_0(\er^4)^*$. Finally, as $\sigma$ behaves asymptotically 
as $\mu_1^1$ at the origin, we conclude that
$$
\int_{B(0,1)\times [-1,1]^2}\Psi(x,y)\,d\sigma_{0,r_j}(x,y)\to
\int_{x_2=y_2=0}\Psi(x,y)\,d\hausd(x,y)
$$
for any test function $\Psi\in \C_c(\er^4)$.

We have shown that the condition of Proposition \ref{p:simon}
is verified at $\sigma$-almost all points $(x,y)\in M^*$.
By Theorem \ref{t:keyc}, it holds at $\hausd$-a.e.\ $(x,y)\in M^*$
and thus $M^*$ is $2$-rectifiable. Getting all together
we conclude that $\Gamma(\Omega)$ is $2$-rectifiable.

\end{proof}

\begin{thm}\label{t:mainpart}
Let $f\colon \Omega\to\er^2$ be a bounded sense-preserving $BV$ homeomorphism
and $\Gamma(\Omega)$ be its graph.
Then $(\Gamma(\Omega),\mu)$ is a $2$-dimensional 1-multiplicity 
rectifiable current.
\end{thm}

\begin{proof}
Consider the orientation making the basis 
$(\tau_1,\tau_2)$ positive on $M$, where
$\tau$ is as in Theorem \ref{t:rect}.
Given a positive parametrization 
$$
\psi=(\psi_1,\psi_2,\psi^1,\psi^2)\colon E\to\Gamma(\Omega),
$$
where $E\subset\er^2$ is measurable,
we need to show that 
\begin{align}
\label{align1}
\int_{E}\det \bigl(\Nabla \psi_1,\Nabla\psi_2\bigr)\,dt&=
\mu_{12}(\psi(E)),\\
\label{align2}
\int_{E}\det \bigl(\Nabla \psi^1,\Nabla\psi^2\bigr)\,dt&=
\mu^{12}(\psi(E)),\\
\label{align3}
\int_{E}\det \bigl(\Nabla \psi_j,\Nabla\psi^i\bigr)\,dt&=
\mu_i^j(\psi(E)),\qquad i,j=1,2.
\end{align}
However, \eqref{align1} is just a change of variables, \eqref{align2}
follows from Theorem \ref{t:det}. Concerning \eqref{align3},
we use the fact that the formula holds for scalar $BV$-functions,
see \cite[4.5.9]{Fe}, \cite[Section 4.1.5]{GMS}.
Let, for example, $i=j=1$. Set $\tilde\psi=(\psi_1,\psi_2,\psi^1)$.
Then $\tilde\psi$ is a Lipschitz mapping to the graph of the 
scalar $BV$-function $f_1$ and thus
$$
\int_{E}\det \bigl(\Nabla \psi_1,\Nabla\psi^1\bigr)\,dt=
-D_2f_1((\psi_1,\psi_2)(E))=\mu_1^1(\psi(E)).
$$
\end{proof}

\begin{thm}\label{t:bdless} 
The current from Theorem \ref{t:mainpart} is boundaryless.
\end{thm}

\begin{proof}
Consider a differential form $\omega\,dx_1=\eta(x)\ff(y)\,dx_2$ where
$\eta\in \C_c^{\infty}(\Omega)$
and $\ff\in \C_c^{\infty}(f(\Omega))$.
Then by the chain rule \cite[Thm.\ 3.96]{AFP} we have
$$
\aligned
\int_{\Gamma(\Omega)}d(\omega\,dx_2)&=
\int_{\Gamma(\Omega)}\eta(x)D_1\ff(y)\,dy_1\,dx_2
+\eta(x)D_2\ff(y)\,dy_2\,dx_2
\\&\qquad+\ff(y)D_1\eta(x)\,dx_1\,dx_2
\\&=\langle D_1f_1, \eta (D_1\ff)\circ f\rangle
+\langle D_1f_2, \eta (D_2\ff)\circ f\rangle
+\int_{\Omega}D_1\eta(x)\ff(f(x))\,dx.
\endaligned
$$
The expression on the right hand side vanishes if
$f$ is smooth (and not necessarily invertible), 
it is the standard integration by parts.
In the general case we can use a routine approximation
argument, as we do not need invertibility at this point.
Note that by mollification we obtain uniform convergence
at the part $(D_i\ff)\circ f$ and weak* convergence of the 
gradients. 
As a next step, we observe that linear combinations of functions
$\omega$ of type $\omega(x,y)=\eta(x)\ff(y)$ are dense in
$\C_c^1(\Omega\times f(\Omega))$, so that we obtain
$$
\int_{\Gamma(\Omega)}d(\omega\,dx_2)=0
$$
for the differential forms of type $\omega(x,y)\,dx_2$
with compact support in $\Gamma(\Omega)$. 
Similarly we handle differential forms $\omega(x,y)\,dx_1$.
For differential forms of type 
$\omega(x,y)\,dy_1$ and $\omega(x,y)\,dy_2$ we pass to
the inverse mapping, as we already know the results
of Theorem~\ref{t:identity}.

\end{proof}

\begin{proof}[Proof of Theorem \ref{t:main}]
The assertions are contained in Theorems \ref{t:rect}, \ref{t:mainpart} and 
\ref{t:bdless}.
\end{proof}

\section{Example}

In this section we construct a $\BV$-homeomorphism 
$f\colon [-1,1]^2\to[-1,1]^2$
such that both $f$ and $f^{-1}$ satisfy the $(N)$-condition,
$J_f>0$ a.e., $J_{f^{-1}}>0$ a.e.,
but neither $f$ nor $f^{-1}$ are Sobolev.

\subsection{Auxiliary mappings}
Set
$$
a_k=\frac{1+2^{1-k}}{2},\quad b_k=2^{1-k}.
$$
Then 
$$
\aligned
&1=a_1>a_2>\dots,\qquad \lim_ka_k=\frac12,\\
&1=b_1>b_2>\dots,\qquad \lim_kb_k=0,
\endaligned
$$ 
\eqn{howb}
$$
b_k=2b_{k+1} \text{\quad and \quad }b_k-b_{k+1}=2(a_k-a_{k+1}),\qquad k=1,2,\dots.
$$
We are going to construct a mapping $g_k$ of
$$
P_k:=[-2^{-k}a_k,\,2^{-k}a_k]\times [-2^{-k}b_k,\,2^{-k}b_k]
$$ 
onto itself.
Denote 
$$
\aligned
Q_k&=[-2^{-k}a_{k+1},\,2^{-k}a_{k+1}]\times[-2^{-k}b_{k+1},\,2^{-k}b_{k+1}],\\
\xi_k(x)&=\frac{a_k-2^k|x_1|}{a_k-a_{k+1}},\\
\eta_k(x)&=\frac{b_k-2^k|x_2|}{b_k-b_{k+1}},\\
A_k&=\bigl\{x\in P_k\setminus Q_k\colon \xi_k(x)\ge \eta_k(x)\bigr\},\\
B_k&=\bigl\{x\in P_k\setminus Q_k\colon \xi_k(x)\le \eta_k(x)\bigr\},\\
\mathbf T_k^t&=
\begin{pmatrix}
0,&\frac{a_{k}+t(a_{k+1}-a_k)}{b_{k}+t(b_{k+1}-b_k)}\\
\frac{b_{k}+t(b_{k+1}-b_k)}{a_{k}+t(a_{k+1}-a_k)},&0
\end{pmatrix}.
\endaligned
$$
We set
$$
\aligned
g_k(x)&=(g_k^1(x),g_k^2(x))
=
\begin{cases}
\mathbf T_k^1x,& x\in Q_k,\\
\mathbf T_{k}^{\min\{\xi_k(x),\eta_k(x)\}}x,& x\in P_k\setminus Q_k.\\
\end{cases}
\endaligned
$$
In particular, $g_k(x)=\mathbf T_k^0x$ on $\partial P_k$.
The function $g_k$ is bi-Lipschitz, indeed, $g_k^{-1}=g_k$. Let us estimate the gradient of $g_k$.
If $x\in A_k$, we have 
$$
\aligned
&t:=\min\{\xi_k(x),\eta_k(x)\}=\eta_k(x),\\
&2^k|x_2|=b_{k}+t(b_{k+1}-b_k),\\
&2^k|x_1|\le a_{k}+t(a_{k+1}-a_k)
\endaligned
$$ 
and 
$$
\aligned
g_k^1(x)&=\frac{a_{k}+t(a_{k+1}-a_k)}{b_{k}+t(b_{k+1}-b_k)}\;x_2=\pm2^{-k}\bigl(a_{k}+t(a_{k+1}-a_k)\bigr),\\
g_k^2(x)&=\frac{b_{k}+t(b_{k+1}-b_k)}{a_{k}+t(a_{k+1}-a_k)}\;x_1=\pm2^k \frac{x_1x_2}{a_{k}+t(a_{k+1}-a_k)}.
\endaligned
$$
It follows
$$
\aligned
|\Nabla g_k^1(x)|&\le 2^{-k}(a_k-a_{k+1})|\Nabla \eta_k(x)|\le \frac{a_k-a_{k+1}}{b_k-b_{k+1}}=\frac12,
\\
|\Nabla g_k^2(x)|&\le 2+2^k\frac{|x_1x_2|(a_k-a_{k+1})}{(a_{k}+t(a_{k+1}-a_k))^2}|\Nabla \eta_k(x)|\\
&\le 2+\frac{b_k(a_k-a_{k+1})}{a_{k+1}(b_k-b_{k+1})}\le 3,
\endaligned
$$
so that 
\eqn{naa}
$$
|\Nabla g_k(x)|\le 4,\qquad x\in A_k.
$$
Since 
$$
|A_k|\le 2^{1-2k}a_k(b_k-b_{k+1})\le 2^{2-3k},
$$
we obtain
\eqn{ona}
$$
\int_{A_k}|\Nabla g_k|\,dx\le 2^{4-3k}.
$$
If $x\in B_k$, we have 
$$
\aligned
&t:=\min\{\xi_k(x),\eta_k(x)\}=\xi_k(x),\\
&2^k|x_2|\le b_{k}+t(b_{k+1}-b_k),\\
&2^k|x_1|= a_{k}+t(a_{k+1}-a_k)
\endaligned
$$ 
and 
$$
\aligned
g_k^1(x)&=\frac{a_{k}+t(a_{k+1}-a_k)}{b_{k}+t(b_{k+1}-b_k)}\;x_2=
\pm 2^k \frac{x_1x_2}{b_{k}+t(b_{k+1}-b_k)}\\
g_k^2(x)&=\frac{b_{k}+t(b_{k+1}-b_k)}{a_{k}+t(a_{k+1}-a_k)}\;x_1=
\pm 2^{-k}\bigl(b_{k}+t(b_{k+1}-b_k)\bigr).
\endaligned
$$
It follows
$$
\aligned
|\Nabla g_k^1(x)|&\le 1+\frac{a_k}{b_{k+1}}+
2^k\frac{|x_1x_2|(b_k-b_{k+1})}{(b_{k}+t(b_{k+1}-b_k))^2}|\Nabla \xi_k(x)|
\\&\le 
1+2^k
+
\frac{a_k(b_k-b_{k+1})}{b_{k+1}(a_k-a_{k+1})}\le 2^{k+2},
\\
|\Nabla g_k^2(x)|&\le  2^{-k}(b_k-b_{k+1})|\Nabla \eta_k(x)|\le \frac{b_k-b_{k+1}}{a_k-a_{k+1}}\le 2,\\
\endaligned
$$
so that
\eqn{nab}
$$
|\Nabla g_k(x)|\le 2^{k+3},\qquad x\in B_k.
$$
Since 
$$
|B_k|\le 2^{1-2k}b_k(a_k-a_{k+1})=2^{1-4k},
$$
we obtain
\eqn{onb}
$$
\int_{B_k}|\Nabla g_k|\,dx\le 2^{4-3k}.
$$
Of course
\eqn{nac}
$$
|\Nabla g_k(x)|=\frac{a_{k+1}}{b_{k+1}},\qquad x\in Q_k.
$$
Since $|Q_k|=4^{1-k}a_{k+1}b_{k+1}$, we have
\eqn{onc}
$$
\int_{Q_k}|\Nabla g_k|\,dx = 4^{1-k} a_{k+1}^2\le 4^{1-k}.
$$

\subsection{The construction of the approximating sequence} 
Consider multiindices $\alpha=(\alpha_1,\dots,\alpha_k)\in\{-1,1\}^k$
and $\beta=(\beta_1,\dots,\beta_k)\in \{-1,1\}^k$, $k=1,2,\dots$.
Then we define 
$$
\aligned
u_{\alpha}=\sum_{j=1}^k2^{-j}\alpha_ja_j,
\quad
v_{\alpha}=\sum_{j=1}^k2^{-j}\alpha_jb_j,
\qquad \alpha\in  \{-1,1\}^k,\\
u_{\beta}=\sum_{j=1}^k2^{-j}\beta_ja_j,
\quad
v_{\beta}=\sum_{j=1}^k2^{-j}\beta_jb_j,\qquad \beta\in  \{-1,1\}^k.
\endaligned
$$
Set 
$$
\aligned
X_{\alpha}&=[u_{\alpha}-2^{-k}a_k,\;u_{\alpha}+2^{-k}a_k],\\
X_{\alpha}'&=[u_{\alpha}-2^{-k}a_{k+1},\;u_{\alpha}+2^{-k}a_{k+1}],
 \qquad&\alpha\in  \{-1,1\}^k,\\
Y_{\beta}&=[v_{\beta}-2^{-k}b_k,\;v_{\beta}+2^{-k}b_k], \\
Y_{\beta}'&=[v_{\beta}-2^{-k}b_{k+1},\;v_{\beta}+2^{-k}b_{k+1}], 
\qquad&\beta\in  \{-1,1\}^k.
\endaligned
$$
Now, we are ready to construct our approximating mappings by gluing the copies
of $g_k$. Namely, set
$$
\aligned
P_{\alpha,\beta}&=X_{\alpha}\times Y_{\beta},
\quad
Q_{\alpha,\beta}&=X_{\alpha}'\times Y_{\beta}',
\qquad \alpha,\beta\in\{-1,1\}^k,
\endaligned
$$
which are  translated copies of $P_k$ and $Q_k$, respectively,
and 
\eqn{howfk}
$$
f_k(x)=(u_{\beta},v_{\alpha})+g_k\big(x-(u_{\alpha},v_{\beta})\big),\qquad 
x\in P_{\alpha,\beta},\ \alpha,\beta\in\{-1,1\}^k.
$$
This defines $f_k$ on 
$$
S_k:=\bigcup_{\alpha,\beta\in\{-1,1\}^k}P_{\alpha,\beta}.
$$
We have $S_k=Q_0=[-1,1]^2$ if $k=1$, otherwise we let $f_k=f_{k-1}$
on $Q_0\setminus S_k$.

\subsection{Sobolev estimates} Since the functions $f_k$ on $P_{\alpha,\beta}$ are 
only translates of $g_k$, we use \eqref{ona}, \eqref{onb} and \eqref{onc} to estimate
$$
\aligned
\int_{S_j\setminus S_{j+1}}|\Nabla f_k|\,dx&=\sum_{\alpha,\beta\in\{-1,1\}^j}
\int_{P_{\alpha,\beta}\setminus Q_{\alpha,\beta}}|\Nabla f_j|\,dx 
=4^k\int_{P_j\setminus Q_j}|\Nabla g_j|\,dx
\\&\le 2^{4-j},\qquad j=1,\dots,k,
\endaligned
$$
and
$$
\aligned
\int_{S_{k+1}}|\Nabla f_k|\,dx&\le \sum_{\alpha,\beta\in\{-1,1\}^k}\int_{Q_{\alpha,\beta}}
|\Nabla f_k|\,dx
\\&\le 4^k4^{1-k}=4.
\endaligned
$$
All together
$$
\int_{Q_0}|\Nabla f_k|\,dx\le 20.
$$

\subsection{Passing to the limit}
We see that
the sequence $(f_k)_k$ is bounded in $W^{1,1}(Q_0)$. At the same time,
the sequence converges in $C(Q_0)$ to a continuous mapping $f$. 
It is obvious from the construction that 
$f_k^{-1}=f_k$ for each $k$, so that $f^{-1}=f$, $f$ is a BV homeomorphism.
We have $|S_k|=4^k|P_k|=4^{k}4^{1-k}a_kb_k\le 2^{3-k}$. The intersection
\eqn{theS}
$$
S=\bigcap_kS_k
$$
is a Cantor set of measure zero. 
Now, the function $f$ is locally Lipschitz
on $Q_0\setminus S$ and maps $S$ to itself, so that $f$ satisfies
the Luzin $(N)$-condition. Also it is obvious that $J_f>0$ a.e.\ in $Q_0\setminus S$,
so simply $J_f>0$ a.e.
Now, denote
$$
\aligned
X&=\bigcap_{k=1}^{\infty}\bigcup_{\alpha\in\{-1,1\}^k}X_{\alpha},\\
Y&=\bigcap_{k=1}^{\infty}\bigcup_{\beta\in\{-1,1\}^k}Y_{\beta}.
\endaligned
$$
Then $S=X\times Y$, 
$$
\aligned
|X|&=\lim_{k\to\infty}\sum_{\alpha\in\{-1,1\}^k}|X_{\alpha}|=2\lim_{k\to\infty}a_k=
1,\\
|Y|&=\lim_{k\to\infty}\sum_{\beta\in\{-1,1\}^k}|Y_{\beta}|=2\lim_{k\to\infty}b_k=0.
\endaligned
$$
Consider a vertical segment $\{z_1\}\times [-1,1]$, where $z_1\in X$.
Then the function 
$x_2\mapsto f(z_1,x_2)$ maps $Y$ onto $X$, so that it fails to satisfy the 
Luzin $(N)$-condition, in particular, it fails to be absolutely continuous.
It follows that $f$ is not a Sobolev mapping, in other words, the singular part
of $Df$ is nontrivial.

\begin{proof}[Proof of Theorem \ref{t:example}] The existence
of a function with required properties follows from our construction.
\end{proof}


\begin{thebibliography}{199}

\bibitem{Alb}
  \by{\name{Alberti}{G.}}
  \paper{Rank one property for derivatives of functions with bounded
		variation}
  \jour{Proc. Roy. Soc. Edinburgh Sect. A}
  \vol{123}
  \pages{239--274}
  \yr{1993}
  \endpaper


\bibitem{AFP}
  \by{\name{Ambrosio}{L.}, \name{Fusco}{N.}\et\ \name{Pallara}{D.}}
  \book{Functions of bounded variation and free discontinuity problems}
  \publ{Oxford Mathematical Monographs.
    The Clarendon Press, Oxford University Press, New York, 2000}
  \endbook
  
\bibitem{Ball}
\by{\name{Ball}{J.\,M.}}
  \paper{Convexity conditions and existence theorems in nonlinear
              elasticity}
  \jour{Arch. Rational Mech. Anal.}
  \vol{63}
  \pages{337--403}
  \yr{1977}
  \endpaper



\bibitem{CHM}
  \by{\name{Cs\"ornyei}{M.}, \name{Hencl}{S.}\et\ \name{Mal\'y}{J.}}
  \paper{Homeomorphisms in the Sobolev space $W^{1,n-1}$}
  \jour{J. Reine Angew. Math}
  \vol{644}
  \pages{221--235}
  \yr{2010}
  \endpaper



\bibitem{DS}
  \by{\name{D'Onofrio}{L.}\et\ \name{Schiattarella}{R.}}
  \paper{On the total variations for the inverse of a BV-homeo\-morphism}
  \jour{Adv. Calc. Var.}
  \vol{6}
  \pages{321--338}
  \yr{2013}
  \endpaper


\bibitem{DHMS}
  \by{\name{D'Onofrio}{L.}, \name{Hencl}{S.}, \name{Mal\'y}{J.}\et\ \name{Schiattarella}{R.}}
  \paper{Note on Lusin $(N)$ condition and the distributional determinant}
  \jour{J. Math. Anal. Appl.}
  \vol{439}
  \pages{171--182}
  \yr{2016}
  \endpaper


 
\bibitem{Fe}
  \by{\name{Federer}{H.}}
  \book{Geometric measure theory}
  \publ{Die Grundlehren der mathematischen Wissenschaften,
    Band 153 Springer-Verlag, New York}
  \yr{Second edition 1996}
  \endbook


\bibitem{FG}
  \by{\name{Fonseca}{I.}\et\ \name{Gangbo}{W.}}
  \book{Degree Theory in Analysis and Applications}
  \publ{Clarendon Press, Oxford, 1995}
  \endbook


\bibitem{GMS}
\by{\name{Giaquinta}{M.}, \name{Modica}{G.}\et\ \name{Sou\v cek}{J.}}
\book{Cartesian currents in the calculus of variations. I., II.}
\publ{Ergebnisse der Mathematik und ihrer Grenzgebiete. 3. Folge. 
A Series of Modern Surveys in Mathematics 
37. Springer-Verlag, Berlin 1998}
\endbook


\bibitem{HKaL}
  \by{\name{Hencl}{S.}, \name{Kauranen}{A.}\et\ \name{Luisto}{R.}}
  \paper{Weak regularity of the inverse under minimal assumptions}
  \jour{preprint arXiv:1804.03449, 2019}
  \endprep
  

\bibitem{HKaM}
  \by{\name{Hencl}{S.}, \name{Kauranen}{A.}\et\ \name{Mal\'y}{J.}}
  \paper{On distributional adjugate and derivative of the inverse}
  \jour{preprint arXiv:1904.04574, 2019}
  \endprep



\bibitem{HK}
  \by{\name{Hencl}{S.}\et\ \name{Koskela}{P.}}
  \paper{Regularity of the inverse of a planar Sobolev homeomorphism}
  \jour{Arch. Rational Mech. Anal}
  \vol{ 180} 
  \yr{2006}
  \pages{75--95}
  \endpaper

\bibitem{HKbook}
  \by{\name{Hencl}{S.}\et\ \name{Koskela}{P.}}
  \book{Lectures on Mappings of finite distortion}
  \publ{Lecture Notes in Mathematics 2096, Springer, 2014, 176pp}
  \endbook






\bibitem{HKM}
  \by{\name{Hencl}{S.}, \name{Koskela}{P.}\et\ \name{Mal\'y}{J.}}
  \paper{Regularity of the inverse of a Sobolev homeomorphism in space}
  \jour{Proc. Roy. Soc. Edinburgh Sect. A}
  \vol{136\nom6}
  \yr{2006}
  \pages{1267--1285}
  \endpaper



\bibitem{HKO}
  \by{\name{Hencl}{S.}, \name{Koskela}{P.}\et\ \name{Onninen}{J.}}
  \paper{Homeomorphisms of bounded variation}
  \jour{Arch. Rational Mech. Anal}
  \vol{186}
  \yr{2007}
  \pages{351--360}
  \endpaper
  
\bibitem{HMPS}
  \by{\name{Hencl}{S.}, \name{Moscariello}{G.}, 
  \name{Passarelli di Napoli}{A.}\et\ \name{Sbordone}{C.}}
  \paper{Bi-Sobolev mappings and elliptic equations in the plane}
  \jour{J. Math. Anal. Appl.}
  \vol{355\nom1.}
  \yr{2009}
  \pages{22--32}
  \endpaper
 
\bibitem{KKM}
\by{\name{Kauhanen}{J.}, \name{Koskela}{P.}\et\ \name{Mal\'y}{J.}}
\paper{On functions with derivatives in a Lorentz space}
\jour {Manuscripta Math.}
\vol {100\nom1}
\yr{1999}
\pages {87--101}
\endpaper




\bibitem{Mbook}
  \by{\name{Maggi}{F.}}
  \book{Sets of finite perimeter and geometric variational problems. An introduction to geometric measure theory}
  \publ{Cambridge Studies in Advanced Mathematics, 135. Cambridge University Press, Cambridge, 2012}
  \endbook        


\bibitem{MSZ}
 \by{\name{Mal\'y}{J.}, \name{Swanson}{D.}\et\ \name{Ziemer}{W.\,P.}}
  \paper{The Co-Area Formula for Sobolev Mappings}
  \jour{Trans. Amer. Math. Soc.}
  \vol{355\nom2}
  \yr{2003}
  \pages{477--492}
  \endpaper

\bibitem{MM}
\by{\name{Mal\'{y}}{J.}\et\ \name{Martio}{O.}}
\paper{Lusin's condition (N) and mappings of the class $W^{1,n}$}
\jour{J.\ reine angew.\ Math.} 
\vol{458} 
\yr{1995}
\pages{19--36}
\endpaper

\bibitem{MMi}
\by{\name{Marcus}{M.}\et\ \name{Mizel}{V.\,J.}}
\paper{Transformations by functions in Sobolev spaces and lower
semicontinuity for parametric variational problems}
\jour{Bull.\ Amer.\ Math.\ Soc.} 
\vol{79\nom4}
\yr{1973}
\pages{790--795}
\endpaper

\bibitem{O}
 \by{\name{Onninen}{J.}}
  \paper{Regularity of the inverse of spatial mappings with finite
         distortion}
  \jour{Calc. Var. Partial Differential Equations}
  \vol{26\nom3}
  \yr{2006}
  \pages{331--341}
  \endpaper


\bibitem{Qui}
  \by{\name{Quittnerov\'a}{K.}}
  \book{Functions of bounded variation of several variables (In Slovak)}
  \publ{Master thesis, Faculty of Mathematics and Physics, Charles University, Prague, 2007}
  \endbook        

\bibitem{RR}
  \by{\name{Rado}{T.}\et\ \name{Reichelderfer}{P.\,V.}}
  \book{Sets of finite perimeter and geometric variational problems. An introduction to geometric measure theory}
  \publ{Cambridge Studies in Advanced Mathematics, 135. Cambridge University Press, Cambridge, 2012}
  \endbook        

\bibitem{R1}
\paper{Some geometrical properties of functions and mappings with generalized
derivatives (Russian)}
\jour{Sibirsk. Math. Zh.}
\by{\name{Reshetnyak}{Yu.\,G.}}
\vol{7}
\yr{1966}
\pages{886--919}
\endpaper
 

\bibitem{Simon}
\by{\name{Simon}{L.}}
\book{Lectures on Geometric Measure Theory}
\publ{Proceedings of the Centre for Mathematical Analysis, Australian
National University, Volume 3, 1983}
\endbook

\end{thebibliography}
\end{document}